\documentclass[12pt]{amsart}
\usepackage{graphicx}
\usepackage{amsmath,amsthm, amsfonts,amssymb, stmaryrd,yfonts,pxfonts,pifont,eufrak, bbm}
\newtheorem{Cor}{Corollary}
 \newtheorem{Lemma}{Lemma}
 
 \newtheorem{ex}{Example}
 \newtheorem{Proposition}{Proposition}
 \theoremstyle{definition}
 
 \theoremstyle{remark}
 \newtheorem{Remark}[Lemma]{Remark}
 \numberwithin{equation}{subsection}

\begin{document}
\title[ALTERATIONS AND REARRANGEMENTS OF A NON-AUTONOMOUS DYNAMICAL SYSTEM ]{ALTERATIONS AND REARRANGEMENTS OF A NON-AUTONOMOUS DYNAMICAL SYSTEM}%
\author[1]{PUNEET SHARMA AND MANISH RAGHAV}
\address[1]{Department of Mathematics, I.I.T. Jodhpur, Nagaur Road, Karwad, Jodhpur-342037, INDIA}%
\email{puneet.iitd@yahoo.com, manishrghv@gmail.com }%


\subjclass[2010]{37B20, 37B55, 54H20}

\keywords{non-autonomous dynamical systems, equicontinuity, transitivity, sensitivity, proximality}

\begin{abstract}
In this paper, we discuss the dynamics of alterations and rearrangements of a non-autonomous dynamical system generated by the family $\mathbb{F}$. We prove that while insertion/deletion of a map in the family $\mathbb{F}$ can disturb the dynamics of a system, the dynamics of the system does not change if the map inserted/deleted is feeble open. In the process, we prove that if the inserted/deleted map is feeble open, the altered system exhibits any form of mixing/sensitivity if and only if the original system exhibits the same. We extend our investigations to properties like equicontinuity, minimality and proximality for the two systems. We prove that any finite rearrangement of a non-autonomous dynamical system preserves the dynamics of original system if the family $\mathbb{F}$ is feeble open. We also give examples to show that the dynamical behavior of a system need be not be preserved under infinite rearrangement.
\end{abstract}
\maketitle

%
%

\section{INTRODUCTION}

Dynamical systems have been long used to investigate various physical processes occurring in nature. The theory has been applied effectively across various disciplines of sciences and engineering and has helped providing solutions to a variety of modern day problems. To name a few, the theory has been applied to address problems like "determining chemical dynamics of a system", "estimating population growth of a species" and "controlling dynamics of various electrical and mechanical systems"\cite{md,cd,pd}. The theory developed pertains to determining the dynamics of a general dynamical system and hence is extremely beneficial for addressing problems across a variety of disciplines. Although most of the problems addressed are modelled using autonomous systems, it is intuitive to believe that better estimates can be obtained for a system when the system is modelled in a non-autonomous setting. As any general model approximating any natural or physical process is non-autonomous in nature, such a modification provides greater insight to the problem and hence results in a better approximation of the original system. Thus, it is important to develop the theory of non-autonomous dynamical systems. As a result, some investigations for such a setting in the discrete case have been made and interesting results have been obtained. While \cite{sk1} investigates the topological entropy when the family $\mathbb{F}$ is equicontinuous or uniformly convergent, \cite{sk2} discusses minimality conditions for a non-autonomous system on a compact Hausdorff space while focussing on the case when the non-autonomous system is defined on a compact interval of the real line. In \cite{jd} authors prove that if $f_n\rightarrow f$, in general there is no relation between chaotic behavior of the non-autonomous system generated by $f_n$ and the chaotic behavior of $f$. In \cite{pm2}, authors investigate the dynamics generated by a uniformly convergent sequence of maps. They give conditions under which the dynamics of a non-autonomous system can be determined by the limiting system. In \cite{pm1} authors investigate a non-autonomous system generated by a finite family of maps. In the process, they study properties like transitivity, weak mixing, topological mixing, existence of periodic points, various forms of sensitivities and Li-Yorke chaos. In \cite{bp} authors investigate properties like weakly mixing, topological mixing, topological entropy and Li-Yorke chaos for the non-autonomous system. Although, many of the questions arising for the dynamics of a non-autonomous system have been answered, many questions are still open and are an interesting point of investigation. For example, how does the dynamics of a system change when a map $f$ is introduced in the family $\mathbb{F}$ at $r$-th position? What is the effect on the dynamics of the system when the map at $k$-th position is deleted from the family $\mathbb{F}$? If the family $\mathbb{F}$ is rearranged to obtain a family $\mathbb{G}$, what is the relation between the dynamics of the original system and dynamics of the rearranged system? Under what conditions is the dynamics of a system preserved under alterations (finite insertions/deletions) or rearrangements?\\

So far, each of the questions posed above are open. In this paper, we investigate the relation between the dynamics of a given system and its alteration (or rearrangement). We prove that while alteration of a system $(X,\mathbb{F})$ by an arbitrary continuous map can disturb the dynamics of $(X,\mathbb{F})$, the dynamics of a system is preserved, when altered by a feeble open map. In the process, we prove that if the system is altered by a feeble open map $f$, various forms of mixing and sensitivity are equivalent for the two systems. We extend our investigations to properties like equicontinuity, minimality and proximality for the two system. In the process, we prove that if the family $\mathbb{F}$ is feeble open, any finite rearrangement of the given system preserves the dynamics of the original system. We also give examples to show that the dynamics of a system need not be preserved under infinite rearrangements. Before we move further, we give some of the basic concepts and definitions required.\\

Let $(X,d)$ be a compact metric space and let $\mathbb{F}=\{f_n: n\in\mathbb{N}\}$ be a family of continuous self maps on $X$. For any initial seed $x_0\in X$, any such family $\mathbb{F}$ generates a non-autonomous dynamical system via the relation $x_{n}= f_n(x_{n-1})$. Let $(X,\mathbb{F})$ denote the non-autonomous dynamical system generated by the family $\mathbb{F}$. For any $x_0\in X$, the set $\{ f_n \circ f_{n-1} \circ \ldots \circ f_1(x_0):n\in\mathbb{N}\}$ defines the orbit of the point $x_0$. For any $k\in\mathbb{N}$, let $\mathbb{F}_k$ denote the truncated family $\{f_n: n\geq k+1\}$. The objective of study of a non-autonomous dynamical system is to investigate the orbit of an arbitrary point $x$ in $X$. For notational convenience, let $\omega^n_{n+k}=f_{n+k}\circ f_{n+k-1}\circ\ldots\circ f_{n+1}$ and $\omega_n(x) = f_n\circ f_{n-1}\circ \ldots \circ f_1(x)$ (the state of the system after $n$ iterations). \\

A point $x$ is called \textit{periodic} for  $(X,\mathbb{F})$ if there exists $n\in\mathbb{N}$ such that $\omega_{nk}(x)=x$ for all $k\in \mathbb{N}$. The least such $n$ is known as the period of the point $x$. A system $(X,\mathbb{F})$ is called feeble open if for any non-empty open set $U$ in $X$, $int(f(U))\neq \phi$ for all $f\in \mathbb{F}$. The system $(X,\mathbb{F})$ is equicontinuous if for each $\epsilon>0$, there exists $\delta>0$ such that $d(x,y)<\delta$ implies $d(\omega_n(x),\omega_n(y))<\epsilon$ for all $n\in\mathbb{N},~~x,y\in X$. The system $(X,\mathbb{F})$ is \textit{transitive} (or $\mathbb{F}$ is transitive) if for each pair of non-empty open sets $U,V$ in $X$, there exists $n \in \mathbb{N}$ such that $\omega_n(U)\bigcap V\neq \phi$. The system $(X,\mathbb{F})$ is said to be \textit{minimal} if every point has a dense orbit. The system $(X,\mathbb{F})$ is said to be \textit{weakly mixing} if for any collection of non-empty open sets $U_1, U_2, V_1,V_2$ in $X$ there exists a natural number $n$ such that $\omega_n(U_i) \bigcap V_i \neq \phi$, $i=1,2$. Equivalently, we say that the system is weakly mixing if $\mathbb{F}\times\mathbb{F}$ is transitive. The system is said to be \textit{topologically mixing} if for every pair of non-empty open sets $U, V$ there exists a natural number $K$ such that $\omega_n(U) \bigcap V \neq \phi$ for all $n \geq K$. The system is said to be \textit{sensitive} if there exists a $\delta>0$ such that for each $x\in X$ and each neighborhood $U$ of $x$, there exists $n\in \mathbb{N}$ such that $diam(\omega_n(U))>\delta$. If there exists $K>0$ such that $diam(\omega_n(U))>\delta$ $~~\forall n\geq K$, then the system is \textit{cofinitely sensitive}. A pair $(x,y)$ is proximal for $(X,\mathbb{F})$ if $\liminf\limits_{n\rightarrow \infty} d(\omega_n(x),\omega_n(y))=0$. See \cite{bc,bs,de} for details.\\

\section{Main Results}

Throughout this section, the maps $f_k$ are assumed to be surjective.
\begin{Proposition}
$(X,\mathbb{F})$ is minimal $\Leftrightarrow$ $(X,\mathbb{F}_k)$ is minimal.
\end{Proposition}

\begin{proof}

Let $(X,\mathbb{F})$ be minimal and let $x\in X$. As each $f_k$ is surjective, $\omega_{k}^{-1}(x)$ is non-empty. Further, as $(X,\mathbb{F})$ is minimal, orbit of any $y\in \omega_{k}^{-1}(x)$ (under $\mathbb{F}$) is dense in $X$. As orbit of $x$ (under $\mathbb{F}_k$) and orbit of $y$ (under $\mathbb{F}$) differ by finitely many points (atmost $k$), denseness of orbit of $y$ (under $\mathbb{F}$) implies denseness of orbit of $x$ (under $\mathbb{F}_k$) and hence $(X,\mathbb{F}_k)$ is minimal.

Conversely let $x\in X$ and $y=\omega_{k}(x)$. As $(X,\mathbb{F}_k)$ is minimal, orbit of $y$ (under $\mathbb{F}_k$) is dense in $X$. Further as orbit of $y$ (under $\mathbb{F}_k$) and orbit of $x$ (under $\mathbb{F}$) differ by finitely many points (atmost $k$), denseness of orbit of $y$ (under $\mathbb{F}_k$) ensures denseness of orbit of $x$ (under $\mathbb{F}$) and hence $(X,\mathbb{F})$ is minimal.
\end{proof}

\begin{Proposition}
$(X,\mathbb{F}_k)$ is equicontinuous $\Leftrightarrow$ $(X,\mathbb{F})$ is equicontinuous.
\end{Proposition}

\begin{proof}
Let $(X,\mathbb{F}_k)$ be equicontinuous and let $\epsilon>0$ be given. As $(X,\mathbb{F}_k)$ is equicontinuous, there exists $\rho>0$ ($\rho<\epsilon$) such that $d(x,y)<\rho$ implies $d(\omega^{k}_n(x),\omega^{k}_n(y))<\epsilon~~\forall~~ n\geq k+1$. Also as the set $\{f_1,f_2\circ f_1,\ldots,f_{k}\circ f_{k-1}\circ\ldots\circ f_1\}$ is finite, there exists $\eta>0$ such that $d(x,y)<\eta$ ensures $d(f_r\circ f_{r-1}\circ\ldots\circ f_1(x),f_r\circ f_{r-1}\circ\ldots\circ f_1(y))<\rho$ for $r\in\{1,2,\ldots,k\}$ or $d(\omega_r(x),\omega_r(y))<\rho$ for $r\in\{1,2,\ldots,k\}$. In particular, $d(x,y)<\eta$ gives $d(\omega_{k}(x),\omega_{k}(y))<\rho$ which further implies $d(\omega^{k}_n(\omega_{k}(x)),\omega^{k}_n(\omega_{k}(y)))<\epsilon~~\forall~~ n\geq k+1$ (by equicontinuity of $(X,\mathbb{F}_k)$) or $d(\omega_n(x),\omega_n(y))<\epsilon$ for all $n\in\mathbb{N}$ and hence $(X,\mathbb{F})$ is equicontinuous.

Conversely, let $(X,\mathbb{F})$ be equicontinuous and let $\epsilon>0$ be given. As $(X,\mathbb{F})$ is equicontinuous, there exists $\delta>0$ such that $d(x,y)<\delta$ ensures $d(\omega_n(x),\omega_n(y))<\epsilon~~\forall n\in\mathbb{N}$. Let $x\in X$ and let $\mathcal{N}_x=\{ S(x,\frac{1}{n}) :n\in\mathbb{N}\}$ be the neighborhood base at $x$. As $\bigcap\limits_{U\in\mathcal{N}_x} \omega_k^{-1}(U)=\omega_k^{-1}(\{x\})$, there exists $U\in\mathcal{N}_x$ such that $\omega_k^{-1}(U)\subset \bigcup\limits_{y\in\omega_k^{-1}(\{x\})} S(y,\delta)$ or there exits $m\in\mathbb{N}$ such that $\omega_k^{-1}(S(x,\frac{1}{m}))\subset \bigcup\limits_{y\in\omega_k^{-1}(\{x\})} S(y,\delta)$. Consequently, if  $d(x,z)<\frac{1}{m}$, for any $u\in \omega_k^{-1}(z)$,  $d(u,y)<\delta$ for some $y\in\omega_k^{-1}(\{x\})$ and hence $d(\omega_n(u),\omega_n(y))<\epsilon$ for all $n\in\mathbb{N}$. As $\omega_k(u)=z, \omega_k(y)=x$ and $\omega^k_n\circ\omega_k=\omega_n$, we obtain $d(\omega^k_n(x),\omega^k_n(z))<\epsilon$ for all $n\geq k+1$ and hence the truncated system is equicontinuous at $x$. As the proof holds for any $x\in X$, $(X,\mathbb{F}_k)$ is equicontinuous and hence equicontinuity is equivalent for the two systems.
\end{proof}

\begin{Remark}
The above proofs establish the equivalence of minimality and equicontinuity for the two systems. While equivalence of minimality for the two systems follow from the fact that denseness of a set is not altered by addition or deletion of finitely many points, equivalence of equicontinuity is established by working on each of the fibres of the inverse function (fibres of $f^{-1}$ are functions $g=f^{-1}|_M$ where $M$ is maximal subset of $X$ such that $f|_M$ is injective). The result is intuitive in nature and extends the fact that addition of finitely many maps cannot generate sensitivity in a non-sensitive system. We now turn our attention towards proximality for the two systems.
\end{Remark}

\begin{Proposition}\label{prox}
If $\mathbb{F}$ is commutative then, (x,y) is proximal for $(X,\mathbb{F}_k) \Rightarrow (x,y)$ is proximal for $(X,\mathbb{F})$. Further if each $f_i$ is bijective then, (x,y) is proximal for $(X,\mathbb{F}) \Rightarrow (x,y)$ is proximal for $(X,\mathbb{F}_k)$.
\end{Proposition}

\begin{proof}
If $(x,y)$ is proximal for $(X,\mathbb{F}_k)$ then there exists a sequence $(n_r)$ of positive integers such that $\lim \limits_{r\rightarrow\infty} d(\omega^k_{n_r}(x),\omega^k_{n_r}(y))=0$. As $X$ is compact, there exists $z\in X$ and a subsequence $(n_{r_l})$ of $(n_r)$ such that $\lim \limits_{l\rightarrow\infty}\omega^k_{n_{r_l}}(x) = \lim \limits_{l\rightarrow\infty}\omega^k_{n_{r_l}}(y)=z$. Thus we get, $f_k\circ f_{k-1}\circ\ldots\circ f_1 (\lim \limits_{l\rightarrow \infty}\omega^k_{n_{r_l}}(x))= f_k\circ f_{k-1}\circ\ldots\circ f_1 (\lim \limits_{l\rightarrow \infty}\omega^k_{n_{r_l}}(z))=f_{k}\circ f_{k-1}\circ\ldots\circ f_1(z)$ or $\lim \limits_{l\rightarrow \infty} f_{k}\circ f_{k-1}\circ\ldots\circ f_1 (\omega^k_{n_{r_l}}(x))= \lim \limits_{l\rightarrow \infty} f_{k}\circ f_{k-1}\circ\ldots\circ f_1 (\omega^k_{n_{r_l}}(y))=f_{k}\circ f_{k-1}\circ\ldots\circ f_1(z)$ (as $f_{k}\circ f_{k-1}\circ\ldots\circ f_1$ is continuous). Consequently, $\lim\limits_{l\rightarrow \infty} \omega_{n_{r_l}}(x)= \lim\limits_{l\rightarrow \infty} \omega_{n_{r_l}}(y)= f_{k}\circ f_{k-1}\circ\ldots\circ f_1(z)$ (as $\mathbb{F}$ is commutative) and hence $(x,y)$ is proximal for $(X,\mathbb{F})$.

Conversely, let $(x,y)$ be proximal for $(X,\mathbb{F})$. Thus, there exists sequence $(n_r)$ of natural numbers such that $\lim \limits_{r\rightarrow\infty} d(\omega_{n_r}(x),\omega_{n_r}(y))=0$. Consequently, there exists a subsequence $(n_{r_l})$ of $(n_r)$ and $z\in X$ such that $\lim \limits_{l\rightarrow\infty}\omega_{n_{r_l}}(x) = \lim \limits_{l\rightarrow\infty}\omega_{n_{r_l}}(y)=z$. As $\omega_{n_{r_l}}=\omega^k_{n_{r_l}}\circ (f_{k}\circ f_{k-1}\circ\ldots\circ f_1)$ and the family $\mathbb{F}$ is commutative, we obtain $\lim \limits_{l\rightarrow \infty} f_{k}\circ f_{k-1}\circ\ldots\circ f_1 (\omega^k_{n_{r_l}}(x))= \lim \limits_{l\rightarrow \infty} f_{k}\circ f_{k-1}\circ\ldots\circ f_1 (\omega^k_{n_{r_l}}(y))=z$ or $f_{k}\circ f_{k-1}\circ\ldots\circ f_1 (\lim \limits_{l\rightarrow \infty}\omega^k_{n_{r_l}}(x))= f_{k}\circ f_{k-1}\circ\ldots\circ f_1 (\lim \limits_{l\rightarrow \infty}\omega^k_{n_{r_l}}(y))=z$ (as $f_{k}\circ f_{k-1}\circ\ldots\circ f_1$ is continuous). As each $f_i$ is bijective, $f_{k}\circ f_{k-1}\circ\ldots\circ f_1$ is bijective and thus we obtain $\lim \limits_{l\rightarrow \infty}\omega^k_{n_{r_l}}(x)= \lim \limits_{l\rightarrow \infty}\omega^k_{n_{r_l}}(y)$ or $(x,y)$ is proximal for $(X,\mathbb{F}_k)$.
\end{proof}

\begin{Remark}
The above proof establishes the equivalence of proximality for the two systems when the family $\mathbb{F}$ is commutative and each $f_k$ is a bijection. While proximality is preserved from $(X,\mathbb{F}_k)$ to $(X,\mathbb{F})$ when the family $\mathbb{F}$ is commutative, the converse is proved under additional assumption of bijectivity of the the maps $f_k$. However, the proof uses only injectivity of the maps $f_k$ and hence the result is true when $\mathbb{F}$ is a commutative family of injective maps. Further, both commutativity and bijectivity (injectivity) are necessary conditions to establish the result and the result does not hold good when either of the conditions imposed is dropped. We now give examples in support of our claim.
\end{Remark}

\begin{ex}\label{prox1}
Let $I$ be the unit interval and let $f:I\rightarrow I$ be piecewise continuous linear map such that $f(0)=0, f(\frac{1}{3})=1, f(\frac{2}{3})=0$ and $f(1)=\frac{2}{3}$. Let $g:I\rightarrow I$ be the defined as

$g(x) = \left\{%
\begin{array}{ll}
            2x  & \text{for x} \in [0, \frac{1}{2}] \\
           2-2x & \text{for x} \in [\frac{1}{2}, 1] \\
\end{array} \right.$

Let $(X,\mathbb{F})$ be the non-autonomous system generated by $\mathbb{F}=\{f,g,g,\ldots\}$. It may be noted that $f$ and $g$ do not commute and hence non-autonomous system generated is non-commutative in nature. As $g(0)=g(1)$, $\{0,1\}$ is a proximal set for $(X,\mathbb{F}_k)$ for any $k\in\mathbb{N}$. However, as $f(0)=0$ and $f(1)=\frac{2}{3}$ are fixed for $g$, the pair is not proximal for $(X,\mathbb{F})$. Thus, commutativity is an essential condition to preserve proximality of a pair (from $(X,\mathbb{F}_k)$ to $(X,\mathbb{F})$).

Further, let $h:I\rightarrow I$ be continuous such that $h(0)=0, h(\frac{2}{3})=\frac{1}{4}$ and $h(1)=1$ and let $(X,\mathbb{F})$ be the non-autonomous system generated by $\mathbb{F}=\{h,g,g,\ldots\}$. It may be noted that the system generated is non-commutative in nature. Further, as the set  $\{0,\frac{2}{3}\}$ is proximal for $(X,\mathbb{F})$ but not for $(X,\mathbb{F}_k)$, the converse does not hold in absence of commutativity.
\end{ex}

%
%
%
%
%

\begin{ex}\label{prox2}
Let $S^1$ be the unit circle and let $f:S^1\rightarrow S^1$ be defined as $f(\theta)=\theta+\pi$. Let $g:S^1\rightarrow S^1$ be defined as

$g(\theta) = \left\{%
\begin{array}{ll}
            \theta  & \text{for}~~\theta \in [0, \pi] \\
           \frac{\theta^2}{\pi}-2\theta+2\pi & \text{for~~} \theta \in [\pi, 2\pi] \\
\end{array} \right.$

Let $(X,\mathbb{F})$ be the non-autonomous system generated by $\mathbb{F}=\{f,g,g,\ldots\}$. It may be noted that both $f$ and $g$ are bijective and hence the non-autonomous system generated is bijective (but non-commutative) in nature. Further, as $f([0,\pi])=[\pi,2\pi]$ and $\pi$ is fixed point (attracting from the right) for $g$, any two points in $[0,\pi]$ are proximal for $(X,\mathbb{F})$. However, as $g$ fixes every point in $[0,\pi]$, the truncated system $(X,\mathbb{F}_k)$ ($k\geq 1$) does not exhibit any proximal pair in $[0,\pi]$.
\end{ex}

\begin{ex}\label{prox3}
Let $I$ be the unit interval and let $f,g:I\rightarrow I$ be defined as

$f(x) = \left\{%
\begin{array}{ll}
            x  & \text{for x} \in [0, \frac{1}{2}] \\
            \frac{4}{3}x-\frac{1}{6} & \text{for x} \in [\frac{1}{2},\frac{7}{8}] \\
           1 & \text{for x} \in [\frac{7}{8}, 1] \\
\end{array} \right.$

$g(x) = \left\{%
\begin{array}{ll}
            -2x+\frac{1}{2}  & \text{for x} \in [0, \frac{1}{4}] \\
            2x-\frac{1}{2} & \text{for x} \in [\frac{1}{4},\frac{1}{2}]\\
           x & \text{for x} \in [\frac{1}{2}, 1] \\
\end{array} \right.$

Let $(X,\mathbb{F})$ be the non-autonomous system generated by $\mathbb{F}=\{f,g,g,\ldots\}$. Then, $f$ and $g$ commute and hence non-autonomous system generated is commutative (but not bijective) in nature. As $f([\frac{7}{8},1])=1$ and $g(x)=x$ for any $x\in [\frac{7}{8},1]$, any pair $(x,y)$ ($x,y \in[\frac{7}{8},1]$) is proximal for $(X,\mathbb{F})$ but fails to be proximal for any truncated system.
\end{ex}

\begin{Remark}
The above examples validate the necessity of the conditions imposed in proposition \ref{prox}. While Example \ref{prox1} establishes the necessity of the commutativity condition for the proposition to hold good, Examples \ref{prox2} and \ref{prox3} prove that commutativity or injectivity alone cannot preserve the proximal pairs in the converse direction. Consequently, both commuativity and injectivity of the maps $f_k$ are necessary for the converse to hold good and hence cannot be dropped.
\end{Remark}

\begin{Proposition}\label{tran}
$(X,\mathbb{F})$ is transitive $\Rightarrow$ $(X,\mathbb{F}_k)$ is transitive. If the family $\mathbb{F}$ is feeble open then $(X,\mathbb{F}_k)$ is transitive $\Rightarrow$ $(X,\mathbb{F})$ is transitive.
\end{Proposition}

\begin{proof}
Let $(X,\mathbb{F})$ be transitive and let $U,V$ be any pair of non-empty open subsets in $X$. As $(X,\mathbb{F})$ is transitive, for the pair $U'=\omega_{k}^{-1}(U),V$ of non-empty open sets in $X$, there exists $r\in\mathbb{N}$ such that $\omega_r(U')\cap V\neq\phi$. Also, transitivity of $(X,\mathbb{F})$ enures that the set $\{r\in\mathbb{N}: \omega_r(U')\cap V\neq\phi\}$ is infinite. Consequently there exists $r>k$ such that $\omega_r(U')\cap V\neq\phi$ or $\omega^k_r(U)\cap V\neq\phi$ and hence $(X,\mathbb{F}_k)$ is transitive.

Let $(X,\mathbb{F}_k)$ be transitive and let $U,V$ be any pair of non-empty open subsets in $X$. As the family $\mathbb{F}$ is feeble open $\omega_{k}(U)$ has a non-empty interior. Thus, for open sets $U'=int(\omega_{k}(U)), V$ in $X$, there exists $r\in\mathbb{N}$ such that $\omega^k_r(U')\cap V\neq\phi$. Consequently, $\omega^k_r(\omega_{k}(U))\cap V\neq\phi$ or $\omega_r(U)\cap V\neq\phi$ and hence $(X,\mathbb{F})$ is transitive..
\end{proof}

\begin{Remark}
The above proof establishes the equivalence of transitivity for the two systems $(X,\mathbb{F})$ and $(X,\mathbb{F}_k)$. Though the property is preserved from $(X,\mathbb{F})$ to $(X,\mathbb{F}_k)$ unconditionally, the proof of the converse holds good when that family $\mathbb{F}$ is feeble open. As absence of feeble openness destroys the topological structure of an open set over iterations, feeble openness is a necessary condition for the converse to hold good. Further, as the proof does not use the structure of open sets explicitly, $U_1,U_2$ interact with $V_1,V_2$ for $(X,\mathbb{F})$ (or $(X,\mathbb{F}_k$)) at $r$-th iterate then $U_1,U_2$ and $V_1,V_2$ interact at $r-k$-th (or $r+k$-th) iterate for $(X,\mathbb{F}_k)$ (or $(X,\mathbb{F})$) and hence weakly mixing is equivalent for the two systems under identical conditions. Further, as the set of times of interaction between open sets $U$ and $V$ for the two systems $(X,\mathbb{F})$ and $(X,\mathbb{F}_k)$ are translate of each other (by constant $k$), the similar proof gives equivalence of topological mixing under identical conditions. We now establish our claims below.
\end{Remark}

\begin{Cor}
$(X,\mathbb{F})$ is weakly mixing (topological mixing) $\Rightarrow$ $(X,\mathbb{F}_k)$ is weakly mixing (topological mixing). If the family $\mathbb{F}$ is feeble open then $(X,\mathbb{F}_k)$ is weakly mixing (topological mixing) $\Rightarrow$ $(X,\mathbb{F})$ is weakly mixing (topological mixing).
\end{Cor}

\begin{proof}
The proof follows from discussions in Remark \ref{sens} and Proposition \ref{tran}.
\end{proof}

\begin{ex}\label{sens}
Let $I$ be the unit interval and let $f,g:I\rightarrow I$ be defined as

$f(x) = \left\{%
\begin{array}{ll}
            0  & \text{for x} \in [0, \frac{1}{2}] \\
           2x-1 & \text{for x} \in [\frac{1}{2}, 1] \\
\end{array} \right.$

$g(x) = \left\{%
\begin{array}{ll}
            2x  & \text{for x} \in [0, \frac{1}{2}] \\
           2-2x & \text{for x} \in [\frac{1}{2}, 1]
\end{array} \right.$ \\

Let $(X,\mathbb{F})$ be the non-autonomous system generated by $\mathbb{F}=\{f,g,g,\ldots\}$. For any $k\in\mathbb{N}$, $(X,\mathbb{F}_k)$ is the autonomous system generated by tent map and hence exhibits all forms of mixing and sensitivities. However for any open set $U$,  $U\subset[0,\frac{1}{2}]$, $\omega_r(U)=\{0\}$ for any $r\in\mathbb{N}$. Thus the non-autonomous system does not exhibit any form of mixing or sensitivity and hence feeble openness is necessary to preserve any form of mixing or sensitivity (from $(X,\mathbb{F}_k)$ to $(X,\mathbb{F})$). We now establish that feeble openness is indeed sufficient to preserve sensitivity from $(X,\mathbb{F}_k)$ to $(X,\mathbb{F})$.
\end{ex}

\begin{Proposition}
$(X,\mathbb{F})$ is sensitive $\Rightarrow$ $(X,\mathbb{F}_k)$ is sensitive. If the family $\mathbb{F}$ is feeble open then
$(X,\mathbb{F}_k)$ is sensitive $\Rightarrow$ $(X,\mathbb{F})$ is sensitive.
\end{Proposition}

\begin{proof}
Let $(X,\mathbb{F})$ be sensitive with $\delta$ as constant of sensitivity. For any open set $U$, continuity of each $f_i$ implies $U'=\omega_{k}^{-1}(U)$ is open and hence there exists $r\in\mathbb{N}$ such that $diam(\omega_r(U'))>\delta$. As the set of times of expansion is infinite for a sensitive system, there exists $m>k$ such that $diam(\omega_m(U'))>\delta$ which implies $diam(\omega^k_m(U))>\delta$ and hence $(X,\mathbb{F}_k)$ is sensitive.

Conversely let $(X,\mathbb{F}_k)$ be sensitive with $\delta$ as constant of sensitivity and let $U$ be a non-empty open set in $X$. As the family $\mathbb{F}$ is feeble open, $U'= int(\omega_{k}(U))$ is non-empty and hence sensitivity of $(X,\mathbb{F}_k)$ yields $m\in\mathbb{N}$ such that $diam(\omega^k_m(U'))>\delta$. Consequently, $diam(\omega^k_m(\omega_{k}(U)))>\delta$ or $diam(\omega_m(U))>\delta$ and hence $(X,\mathbb{F})$ is sensitive.
\end{proof}

\begin{Remark}
The above proof establishes equivalence of sensitivity for the two systems $(X,\mathbb{F})$ and $(X,\mathbb{F}_k)$ under feeble openness of the family $\mathbb{F}$.  Once again, while sensitivity of $(X,\mathbb{F})$ implies sensitivity of $(X,\mathbb{F}_k)$ unconditionally, the converse is true when the family $\mathbb{F}$ is feeble open. As noted in Example \ref{sens}, feeble openness is necessary for the converse to hold good and hence cannot be dropped. Further, it may be noted that if one of the systems is sensitive with sensitivity constant $\delta$, then the proof establishes the sensitivity of the other system with same constant of sensitivity and hence the two systems are sensitive with same sensitivity constant. Finally, as the times of expansion (of an open set $U$) for the two systems are translate (by constant $k$) of each other, a similar proof establishes the equivalence of syndetic (cofinite) sensitivity for the two systems. Hence we get the following corollary.
\end{Remark}

\begin{Cor}
$(X,\mathbb{F})$ is syndetically (cofinitely) sensitive $\Rightarrow$ $(X,\mathbb{F}_k)$ is syndetically (confinitely) sensitive. If the family $\mathbb{F}$ is feeble open then $(X,\mathbb{F}_k)$ is syndetically (cofinitely) sensitive $\Rightarrow$ $(X,\mathbb{F})$ is syndetically (cofinitely) sensitive.
\end{Cor}

\begin{Remark}
The proofs above establish that for a feeble open family $\mathbb{F}$, $(X,\mathbb{F})$ exhibits any form of mixing (sensitivity) if and only if $(X,\mathbb{F}_k)$ also exhibits the same. It may be noted that if $(X,\mathbb{G})$ is a finite rearrangement of $(X,\mathbb{F})$ then there exists $k\in\mathbb{N}$ such that $\mathbb{G}_k=\mathbb{F}_k$. Consequently, for a feeble open family $\mathbb{F}$, as $(X,\mathbb{F})$ and $(X,\mathbb{F}_k)$ (and similarly $(X,\mathbb{G})$ and $(X,\mathbb{G}_k)$) exhibit identical notions of mixing (sensitivity), $(X,\mathbb{F})$ exhibits any form of mixing (sensitivity) if and only if $(X,\mathbb{G})$ exhibits identical form of mixing (sensitivity) and hence various notions of mixing (sensitivity) are preserved under finite rearrangements. Further, it may be noted that as minimality and equicontinuity are equivalent for two systems $(X,\mathbb{F})$ and $(X,\mathbb{F}_k)$ unconditionally, the notions of minimality and equicontinuity are preserved under finite rearrangements. Hence we obtain the following corollaries.
\end{Remark}

\begin{Cor}
Let $(X,\mathbb{F})$ be a non-autonomous dynamical system and let $\mathbb{G}$ be a finite rearrangement of $\mathbb{F}$. Then, $(X,\mathbb{F})$ is minimal (equicontinuous) $\Leftrightarrow$ $(X,\mathbb{G})$ is minimal (equicontinuous).
\end{Cor}

\begin{Cor}
Let $\mathbb{F}$ be feeble open and let $\mathbb{G}$ be a finite rearrangement of $\mathbb{F}$. Then, $(X,\mathbb{F})$ exhibits any notion of mixing (sensitivity) if and only if $(X,\mathbb{G})$ exhibits identical notion of mixing (sensitivity).
\end{Cor}

\begin{Cor}
Let $\mathbb{F}$ be commutative family of bijective maps and let $\mathbb{G}$ be a finite rearrangement of $\mathbb{F}$. Then, $(x,y)$ is proximal for $(X,\mathbb{F})$ $\Leftrightarrow$ $(x,y)$ is proximal for $(X,\mathbb{G})$.
\end{Cor}

\begin{Remark}
The above results derive sufficient conditions under which a dynamical notion is preserved under finite rearrangement. Consequently, while minimality and equicontinuity are preserved unconditionally, various notions of mixing (sensitivity) are preserved when the family $\mathbb{F}$ is feeble open. However, the result is true when $\mathbb{G}$ is a finite rearrangement of $\mathbb{F}$ and the dynamical notions discussed need not be preserved under the stated conditions when the rearrangement $\mathbb{G}$ is an infinite rearrangement. We now an give example to establish our claim.
\end{Remark}

\begin{ex}
Let $X =\{0,1\}^{\mathbb{Z}}$ be the collection of two-sided
sequences of $0$ and $1$ endowed with the product topology. Let
$\sigma : X \rightarrow X$ be defined as
$\sigma(\ldots x_{-2}x_{-1}.x_0 x_1 x_2\ldots)= (\ldots
x_{-2}x_{-1}x_0. x_1 x_2\ldots)$. The map $\sigma$ is the shift operator and
is continuous with respect to the product topology on $X$. Let
$\mathbb{F}=\{\sigma, \sigma^{-1},\sigma,\sigma,\sigma^{-1},\sigma^{-1},\ldots\}$. Thus, the family $\mathbb{F}$ is defined by defining $f_i=\sigma$ when $n(n+1)+1\leq i\leq (n+1)^2$ and $f_i=\sigma^{-1}$ when $(n+1)^2+1\leq i\leq (n+1)(n+2)$. Then, as $\omega_{n(n+1)}(x)=x$ and $\omega_{n(n+1)+r}(x)=\sigma^r(x)$ for $1\leq r\leq n+1$, for any open set $U$ we obtain, $\omega_{n(n+1)+r}(U)=\sigma^r(U)$ for $1\leq r\leq n+1$ and hence the system $(X,\mathbb{F})$ exhibits all forms of mixing and sensitivity. However, as there are equal number of $\sigma$ and $\sigma^{-1}$ between $f_{n(n+1)}$ and $f_{(n+1)(n+2)}$ ($n$ each), the family $\mathbb{F}$ can be rearranged to obtain $\mathbb{G}=\{\sigma, \sigma^{-1},\sigma, \sigma^{-1},\ldots\}$. As $(X,\mathbb{G})$ does not exhibit any form of mixing or sensitivity, any form of mixing or sensitivity need not be preserved under infinite rearrangement. Further, it may be noted that $(X,\mathbb{F})$ is strongly sensitive and hence is not equicontinuous. However, as orbit of any $x$ in $(X,\mathbb{G})$ is $\{x,\sigma(x)\}$, the system $(X,\mathbb{G})$ is equicontinuous and hence equicontinuity is not preserved under infinite rearrangements even when the maps $f_i$ are bijective. Hence the conditions under which the dynamical behavior is preserved for finite rearrangements strictly work when $\mathbb{G}$ is a finite rearrangement and need not preserve the dynamics when the family $\mathbb{F}$ is infinitely rearranged.
\end{ex}

%
%
%
%

\section{Conclusion}

In this work, we investigated the dynamics arising from various possible alterations and rearrangements arising from a given non-autonomous dynamical system $(X,\mathbb{F})$. We prove that if $(X,\mathbb{G})$ is obtained by inserting/deleting finitely many maps from the family $\mathbb{F}$, under certain conditions, the modified system exhibits behavior similar to $(X,\mathbb{F})$ and hence the dynamics is preserved under such modifications. We prove that while minimality and equicontinuity are preserved unconditionally, proximality is preserved when the family $\mathbb{F}$ is commutative and injective. We prove that various notions of mixing and sensitivities are equivalent for the two systems when the family $\mathbb{F}$ is feeble open. We prove that the results established do not hold good when the conditions imposed are relaxed and hence the conditions imposed are indeed necessary for the results to hold good. We generalize our results to the case when the family $\mathbb{G}$ is a finite rearrangement of $\mathbb{F}$. We prove that the results obtained hold good strictly for finite rearrangements and fail to hold true when the rearrangement is infinite.

\bibliography{xbib}

\end{document}